\documentclass[10pt]{amsart}
\usepackage[usenames]{color}
\usepackage{amssymb, latexsym}
\usepackage[neverdecrease]{paralist}
\usepackage[all,2cell,ps]{xy}
\catcode`Œ=\active \def Œ{\'{S}} \theoremstyle{plain}
\newtheorem{thm}{Theorem}[section]
\newtheorem{cor}[thm]{Corollary}
\newtheorem{lm}[thm]{Lemma}

\newtheorem{pr}[thm]{Problem}

\theoremstyle{definition}
\newtheorem{de}[thm]{Definition}
\newtheorem{exm}[thm]{Example}

\newcommand{\p}{\mathcal{P}}

\newcommand{\wpf}{\mathcal{P}_{>0}^{<\om}}
\newcommand{\Om}{\Omega}
\newcommand{\om}{\omega}
\newcommand{\cc}{\mathcal{C}}

\newcommand{\vv}{\mathcal{V}}

\newcommand{\mm}{\mathcal{S}}

\newcommand{\typ}{\mho}

\def\={ \approx  }

\newcommand{\Con}{\textup{Con}\kern1pt}
\newcommand{\Clo}{\textup{Clo}\kern1pt}
\newcommand{\id}{\textup{id}}
\newcommand{\fic}{\textup{fi}}
\numberwithin{equation}{section}
\let\emptyset\varnothing

\numberwithin{equation}{thm}
\begin{document}

\title{Closure operators on algebras}

\author{A. Pilitowska$^1$}
\author{A. Zamojska-Dzienio$^{2,\ast}$}

\thanks{$^{\ast}$ Corresponding author.}

\address{Faculty of Mathematics and Information Science\\
Warsaw University of Technology, ul. Koszykowa 75, 00-662 Warsaw,
Poland}

\email{$^1$apili@mini.pw.edu.pl}

\email{$^2$A.Zamojska-Dzienio@mini.pw.edu.pl,
A.Zamojska@elka.pw.edu.pl}

\urladdr{$^1$ http://www.mini.pw.edu.pl/ \textasciitilde apili}
\urladdr{$^2$ http://www.mini.pw.edu.pl/ \textasciitilde azamojsk}
\thanks{While working on this paper, the authors were supported by the
Statutory Grant of Warsaw University of Technology
504P/1120/0025/000.}


\date{\today}

\begin{abstract}
We study connections between closure operators on an algebra $(A,\Om)$ and congruences on the extended
power algebra defined on the same algebra. We use these connections to give an alternative
description of the lattice of all subvarieties of semilattice
ordered algebras.
\end{abstract}

\maketitle
\section{Introduction}\label{sec1}

In this paper we study closure operators on algebras. M. Novotn\'y \cite{Nov} and J.~Fuchs \cite{Fu} introduced the notion of \emph{admissible closure operators} for monoids and applied them to study some Galois connections in language theory. In \cite{KP05} M.~Ku\v{r}il and
L.~Pol\'{a}k used such operators to describe the varieties of semilattice ordered
semigroups. In this paper we use a similar approach but apply it
to arbitrary algebras. In
particular, we use concepts from the theory of (extended) power algebras
(see e.g. C.~Brink \cite{B93}). This allows us to generalize and simplify methods presented in \cite{KP05}.

In a natural way one can "lift" any operation defined on a set $A$
to an operation on the set of all non-empty subsets of $A$ and
obtain from any algebra $(A,\Om)$ its \emph{power algebra} of
subsets. Power algebras of non-empty subsets with the additional
operation of set-theoretical union are known as \emph{extended power
algebras}. Extended power algebras play a crucial r\^ole in the theory of \emph{semilattice ordered algebras}.

The main aim of this paper is to investigate how congruences on the
extended power algebra on $(A,\Om)$ and closure operators defined on
$(A,\Om)$ are related. The correspondence obtained allows us to give an alternative
(partially simpler) description of the subvariety lattice of semilattice ordered algebras, comparing
to \cite{PZ12a}.

The paper is organized as follows. In Section 2 we provide basic
definitions and results, among others the notion of semilattice ordered algebras. In Section 3 we investigate the connection between
congruences on the extended power algebra defined on an algebra
$(A,\Om)$ and some special closure operators on the same algebra. In
particular, in Theorem \ref{sec4:cor2} we get a one-to-one
correspondence between the set of all fully invariant congruences
and closure operators satisfying additional conditions. In Section 4
we shortly recall the description of free semilattice ordered
algebras given in \cite{PZ12a} and then in Theorem \ref{varlin} we present a new
characterization of linear varieties. Next we apply results obtained
in Section 3 to describe the lattice of varieties of all semilattice
ordered algebras in Theorem \ref{sec6:thm2}. In Section 5 we show
that there is a correspondence between the set of some special
subvarieties of semilattice ordered $\vv$-algebras and the set of
all closure operators on the $\vv$-free
algebra which satisfy additional conditions (Corollary
\ref{sec4:cor5}). In Section 6 we apply the results to some particular
idempotent varieties of algebras and conclude the paper with a
discussion and open questions.

The notation $t(x_1,\dots,x_n)$ means that the term $t$ of the
language of a variety $\vv$ contains no other variables than
$x_1,\dots,x_n$ (but not necessarily all of them). All operations
considered in the paper are supposed to be finitary. We are
interested here only in varieties of algebras, so the notation
$\mathcal{W}\subseteq\mathcal{V}$ means that $\mathcal{W}$ is a
subvariety of a variety $\mathcal{V}$.

\section{Semilattice ordered algebras}\label{sec2}

Let $\typ$ be the variety of all algebras $(A,\Om)$ of a (fixed)
finitary type $\tau\colon\Om\to \mathbb{N}$ and let
$\vv\subseteq\typ$ be a subvariety of $\typ$. In \cite{PZ12a} we
introduced the following definition of a semilattice ordered
algebra.
\begin{de}\label{sec2:def1}
An algebra $(A,\Om,+)$ is called a \emph{semilattice ordered
$\vv$-algebra} (or briefly \emph{semilattice ordered algebra}) if
$(A,\Om)$ belongs to a variety $\vv$, $(A,+)$ is a (join)
semilattice (with semilattice order $\leq$, i.e. $x\leq y\;
\Leftrightarrow \; x+y=y$), and the operations from the set $\Om$
\emph{distribute} over the operation $+$, i.e. for each $0\neq
n$-ary operation $\om\in \Om$, and $x_1,\ldots,x_i,y_i,\dots,x_n\in
A$
\begin{eqnarray*}
\om(x_1,\ldots,x_i+ y_i,\ldots,x_n)=\label{sec2:eq1}
\\
\om(x_1,\ldots,x_i,\ldots,x_n)+
\om(x_1,\ldots,y_i,\ldots,x_n),\nonumber
\end{eqnarray*}
for any $1\leq i \leq n$.
\end{de}

It is easy to notice that in semilattice ordered algebras all
non-constant $\Om$-operations are \emph{monotone} with respect to
the semilattice order $\leq$. So these algebras are \emph{ordered algebras} in the sense of \cite{CL83}
(see also \cite{F66} and \cite{B76}). Basic examples are given by
additively idempotent semirings, distributive lattices, semilattice
ordered semigroups \cite{GPZ05}, modals \cite{RS85} and semilattice modes \cite{K95}.

Each semilattice ordered algebra $(A,\Om,+)$ has the following two
elementary properties for any elements $x_i\leq y_i$ and $x_{ij}$ of
$A$ for $1\leq i\leq n$, $1\leq j\leq r$ and non-constant $n$-ary
operation $\om\in \Om$:
\begin{align}
&\om(x_1,\ldots, x_n)\leq
\om(y_1,\ldots, y_n), \label{monsub} \\
&\om(x_{11},\ldots, x_{n1})+ \ldots +
\om(x_{1r},\ldots, x_{nr})\label{consub}\\
&\leq \om(x_{11}+ \ldots + x_{1r},\ldots,x_{n1}+\ldots + x_{n
r}).\nonumber
\end{align}
Note that in general both \eqref{monsub} and \eqref{consub} hold
also for all derived operations.

The most important examples from the point of view of applications
are given by the following.
\begin{exm}\label{sec3:exm3}{\bf Extended power
algebras of algebras.}\cite{PZ12a} For a given set $A$ denote by
$\mathcal{P}_{>0}A$ the family of all non-empty subsets of $A$. For
any $n$-ary operation $\om\colon A^n\to A$ we define the
\emph{complex operation} $\om\colon(\mathcal{P}_{>0}A)^n\to
\mathcal{P}_{>0}A$ in the following way:
\begin{equation*}
\om(A_1,\dots,A_n):=\{\om(a_1, \dots, a_n)\mid a_i \in A_i\},
\end{equation*}
where $\emptyset\neq A_1,\dots,A_n\subseteq A$. The set
$\mathcal{P}_{>0}A$ also carries a join semilattice structure under
the set-theoretical union $\cup$. In \cite{JT51} B. J${\rm
\acute{o}}$nsson and A. Tarski proved that complex operations
distribute over the union $\cup$. Hence, for any algebra
$(A,\Om)\in\typ$, the \emph{extended power algebra}
$(\mathcal{P}_{>0}A,\Om,\cup)$ is a semilattice ordered
$\typ$-algebra. The algebra $(\wpf A,\Om,\cup)$ of all finite
non-empty subsets of $A$ is a subalgebra of
$(\mathcal{P}_{>0}A,\Om,\cup)$.
\end{exm}

Example \ref{sec3:exm3} is the special case of the next one.

\begin{exm}\label{sec2:exm3}{\bf Extended power algebras of
relation structures.}\cite{JT51} For a given set $A$ denote by
$\mathcal{P}A$ the family of all subsets of $A$. For any $(n+1)$-ary
relation $R\subseteq A^{n+1}$ on a set $A$ we define the $n$-ary
operation $f_R\colon(\mathcal{P}A)^n\to \mathcal{P}A$ in the
following way:
\begin{equation*}
f_R(A_1,\dots,A_n):=\{y\in A\mid \exists(a_1\in A_1), \dots,
\exists(a_n\in A_n)\; R(a_1,\ldots,a_n,y)\}.
\end{equation*}

By results of B. J${\rm \acute{o}}$nsson and A.
Tarski \cite{JT51} such operations distribute over
the set union $\cup$. Thus, for any relation structure $(A,\mathcal{R})$, the
\emph{extended power algebra} $(\mathcal{P}A,\{f_R\mid R\in \mathcal{R}\},\cup)$ of $(A,\mathcal{R})$, is a
semilattice ordered algebra. The algebra $(\mathcal{P}^{<\om} A,\{f_R\mid R\in \mathcal{R}\},\cup)$
of all finite subsets of $A$ is a subalgebra of
$(\mathcal{P}A,\{f_R\mid R\in \mathcal{R}\},\cup)$.
\end{exm}
Now we provide three other examples of semilattice ordered algebras
which are given by closure operators defined on algebras.

\begin{exm}\label{sec3:exm4}{\bf Closed subsets.}
Let $(A,\Om)\in\typ$ be an algebra without constant operations and
let $\cc\colon\p A\to\p A$ be a closure operator on $A$ such that
for each $n$-ary operation $\om\in \Om$ and any subsets
$A_1,\ldots,A_n\in \p A$
\begin{equation}\label{sec3:eq2}
\om(\cc(A_1),\ldots,\cc(A_n))\subseteq \cc(\om(A_1,\ldots,A_{n})).
\end{equation}
Let $\p_{\cc} A:=\{\cc(B)\mid \emptyset\neq B\subseteq A\}$
be the set of all non-empty $\cc$-closed subsets of $A$. For each $n$-ary
$\omega\in \Om$ define an operation $\omega_{\cc}\colon(\p_{\cc}
A)^n\to \p_{\cc} A$ in the following way
\begin{equation*}
\omega_{\cc}(X_1,\ldots,X_n):=\cc(\omega(X_1,\ldots,X_n)),
\end{equation*}
for $X_1,\ldots,X_n\in  \p_{\cc} A$. Moreover, define a binary operation
$+\colon(\p_{\cc} A)^2\to \p_{\cc} A$
\begin{equation*}
X+Y:=\cc(X\cup Y),
\end{equation*}
for $X,Y\in \p_{\cc} A$.

It is evident that $(\p_{\cc} A,+)$ is a semilattice. Note also that for any subsets $B_1,\ldots,B_n\in \p A$
\begin{align*}
&\cc(\omega(\cc(B_1),\ldots,\cc(B_n)))\stackrel{\eqref{sec3:eq2}}{\subseteq}\cc(\cc(\omega(B_1,\ldots,B_n)))=\\
&\cc(\omega(B_1,\ldots,B_n))\stackrel{\eqref{monsub}}{\subseteq}\cc(\omega(\cc(B_1),\ldots,\cc(B_n))).
\end{align*}
Moreover,
\begin{align*}
&\cc(B_1\cup B_2)\subseteq\cc(\cc(B_1)\cup\cc(B_2))\subseteq\cc(\cc(B_1\cup B_2))=\cc(B_1\cup B_2).
\end{align*}
This implies
\begin{equation}\label{ex3:eq2}
\omega_{\cc}(\cc(B_1),\ldots,\cc(B_n))=\cc(\omega(B_1,\ldots,B_n)),
\end{equation}
and
\begin{equation}\label{ex3:eq3}
\cc(B_1)+\cc(B_2)=\cc(B_1\cup B_2).
\end{equation}
Hence, for $n>0$ and any subsets
$X_1,\ldots,X_i,Y_i,\ldots,X_n \in \p_{\cc} A$
\begin{align*}
&\omega_{\cc}(X_1,\ldots,X_i+Y_i,\ldots,X_n)=\\
&\omega_{\cc}(\cc(X_1),\ldots,\cc(X_i\cup Y_i),\ldots,\cc(X_n))\stackrel{\eqref{ex3:eq2}}{=}\\
&\cc(\omega(X_1,\ldots,X_i\cup Y_i,\ldots,X_n))=\\
&\cc(\omega(X_1,\ldots,X_i,\ldots,X_n)\cup\omega(X_1,\ldots,Y_i,\ldots,X_n))=\\
&\cc(\omega(\cc(X_1),\ldots,\cc(X_i),\ldots,\cc(X_n))\cup\omega(\cc(X_1),\ldots,\cc(Y_i),\ldots,\cc(X_n)))\stackrel{\eqref{ex3:eq3}}{=}\\
&\cc(\omega(\cc(X_1),\ldots,\cc(X_i),\ldots,\cc(X_n)))+\cc(\omega(\cc(X_1),\ldots,\cc(Y_i),\ldots,\cc(X_n)))=\\
&\omega_{\cc}(X_1,\ldots,X_i,\ldots,X_n)+
\omega_{\cc}(X_1,\ldots,Y_i,\ldots,X_n)).
\end{align*}

Then $(\p_{\cc} A,\Om_{\cc},+)$, where
$\Om_{\cc}:=\{\omega_{\cc}\mid \omega\in \Om\}$, is a semilattice
ordered $\typ$-algebra.

The algebra $(\p_{fin\cc} A,\Om_{\cc},+)$, where $\p_{fin\cc} A$ is
the set of all non-empty, finitely generated $\cc$-closed subsets of
$A$, is its subalgebra.
\end{exm}

As a special instance of Example \ref{sec3:exm4} one gets the
following.
\begin{exm}\label{sec3:exm5}{\bf $\Gamma$-sinks.}
Let $(A,\Om)\in\typ$ be an algebra  without constant operations and
let $\Gamma \subseteq \Om$. A subalgebra $(S,\Om)$ of $(A,\Om)$ is
said to be a $\Gamma$-\emph{sink} of $(A,\Om)$ if for each $0<n$-ary
operation $\nu\in \Gamma$ and $i=1,\ldots, n$,
$$\nu(A,\ldots,\underbrace{S}_{i},\ldots,A)\subseteq S.$$

In particular, an $\Om$-sink is a \emph{sink} of $(A,\Om)$, as
defined in \cite[Section 3.6]{RS85} and an $\emptyset$-sink is a
subalgebra of $(A,\Om)$.

Denote by $\langle X \rangle_{\Gamma}$ the $\Gamma$-sink generated
by a non-empty set $X\subseteq A$, i.e. the intersection of all
$\Gamma$-sinks of $(A,\Om)$ that include $X$. The mapping
$\cc_{\Gamma}\colon\p A\to \p A$, $X\mapsto \langle X
\rangle_{\Gamma}$ is a closure operator defined on $A$ (see also
\cite{PZ12}).

Let $\p_{\Gamma}(A)$ denote the set of all non-empty $\Gamma$-sinks
of $(A,\Om)$. An \emph{idempotent} (in the sense that each
singleton is a subalgebra) and \emph{entropic} algebra (any two of
its operations commute) is called a \emph{mode} (see \cite{RS02}).
It was shown in \cite{PZ12} that in a mode $(A,\Om)$ for each
$n$-ary complex operation $\om\in \Om$, $\Gamma\subseteq\Om$ and any
non-empty subsets $X_1,\ldots,X_n$ of $A$, $\om(\langle
X_1\rangle_{\Gamma},\ldots, \langle X_n\rangle_{\Gamma})$ is a
$\Gamma$-sink and additionally
\begin{equation*}
\langle\om(X_1,\ldots,X_n)\rangle_{\Gamma} = \om(\langle
X_1\rangle_{\Gamma},\ldots, \langle X_n\rangle_{\Gamma}).
\end{equation*}
Then, $(\p_{\Gamma}(A),\Om_{\cc_{\Gamma}},+)$ is a semilattice
ordered $\typ$-algebra of non-empty $\Gamma$-sinks of $(A,\Om)$. In
particular, for $\Gamma=\emptyset$ one obtains a
semilattice ordered $\typ$-algebra of non-empty subalgebras of
$(A,\Om)$.
\end{exm}
\begin{exm}{\bf Closure algebras.}\cite{JT51} A \emph{closure
algebra} $(S,c,+)$ is a join-se\-mi\-lat\-tice $(S,+)$ with the smallest
element $0$ endowed with one unary operation $c:S\to S$ satisfying:
$c(0)=0$, $c(s_1+s_2)=c(s_1)+c(s_2)$ and $s\leq c(s)=c(c(s))$.
Clearly, $(S,c,+)$ is a semilattice-ordered unary algebra.

The operation of topological closure makes the power set of any topological space into a closure algebra.
\end{exm}

\section{Congruences versus closure operators}\label{sec3}

Let  $A$ be a set and let $\Theta$ be an equivalence relation on the
set $\wpf{A}$. Define an operator $\cc_{\Theta}\colon\p A\to \p A$,
\begin{align*}
\cc_{\Theta}(T)\colon=\{a\in A\mid \exists (U\in \wpf{T})\;\; U\cup
\{a\}\mathrel{\Theta} U\},\;\;{\rm for} \;\; T\subseteq A.
\end{align*}

Conversely, let $\cc\colon\p A\to\p A$ be an 
operator on $A$. We define a binary relation $\Upsilon_
{\cc}\subseteq \wpf A\times\wpf A$, associated with $\cc$, in the
following way
\begin{equation*}
(Q,R)\in \Upsilon_{\cc} \; \Leftrightarrow \; \cc(Q)=\cc(R).
\end{equation*}
It is clear that $\Upsilon_{\cc}$ is an equivalence relation.

\begin{exm}
\begin{enumerate}
\item Let $A$ be a set. For the identity relation $\mathbf{0}_{\wpf A}=\{(X,X)\colon X\in\wpf A\}$ and
the full relation $\mathbf{1}_{\wpf A}=\wpf A\times \wpf A$,
\begin{equation*}
\cc_{\mathbf{0}_{\wpf A}}=\id_{\p A}\;\; {\rm and} \;\;
\cc_{\mathbf{1}_{\wpf A}}(T)=A, \;{\rm for \; any}\; T\subseteq A.
\end{equation*}

On the other hand, for the closure operator $\id_{\p A}$ and the closure
operator $\cc\colon\p A\to \p A$, defined as $\cc(T)=A$, for each
$T\in\p A$, the corresponding relations are, respectively, the identity
relation $\mathbf{0}_{\wpf A}$ and the full relation $\mathbf{1}_{\wpf A}$.

\item Let $A$ be a set and $\theta$ be a binary relation on $A$. We
can define a binary relation $\Theta$ on $\p A$ as follows: for
$X,Y\subseteq A$
\begin{equation*}
X\mathrel{\Theta} Y \;\;\; \Leftrightarrow \;\;\; (\forall x\in
X)(\exists y\in Y)\; x\mathrel{\theta} y \;\; \text{and} \;\;
(\forall y\in Y)(\exists x\in X)\; x\mathrel{\theta} y.
\end{equation*}
By results of C. Brink \cite{B93} it is known that if $\theta$ is an
equivalence relation on a set $A$, then also the relation $\Theta$
is an equivalence relation on the set $\wpf A$. In such case, for
any $T\in \p A$
\begin{equation*}
\cc_{\Theta}(T)=\{a\in A\mid \exists (t\in T)\;\; a\mathrel{\theta}
t\}=\bigcup\limits_{t\in T}t/\theta.
\end{equation*}
Note also that $\Upsilon_{\cc_{\Theta}}=\Theta$.
\end{enumerate}
\end{exm}

Recall that a closure operator $\cc$ on a set $A$ is
\emph{algebraic} if $\cc(B)=\bigcup\{\cc(F)\mid F\in \wpf B \}$ for
any $B\subseteq A$. We say $\cc$ is an \emph{$\varnothing$-preserving} closure operator
if it has  the property $\cc(\emptyset)=\emptyset$.

\begin{lm}\label{sec5:lem10}
Let $A$ be a set and $\Theta$ be a congruence relation on
$(\wpf{A},\cup)$. Then $\cc_{\Theta}\colon\p A\to \p A$ is an
algebraic $\varnothing$-preserving closure operator.
\end{lm}

\begin{proof}
Let $T\subseteq A$. It is easy to notice that
$\cc_{\Theta}(\emptyset)=\emptyset$, $T\subseteq \cc_{\Theta}(T)$
and for $T\subseteq W\subseteq A$, also
$\cc_{\Theta}(T)\subseteq\cc_{\Theta}(W)$.  It implies
$\cc_{\Theta}(T)\subseteq\cc_{\Theta}(\cc_{\Theta}(T))$.

To prove the opposite inclusion, let $a\in
\cc_{\Theta}(\cc_{\Theta}(T))$. Then there is a finite subset $U$ of
$\cc_{\Theta}(T)$ such that $U\mathrel{\Theta} U\cup \{a\}$. Let
$U=\{u_1,u_2,\ldots, u_k\}$. For each element $u_i$ there is a
finite subset $U_i$ of the set $T$ such that $U_i\mathrel{\Theta}
U_i\cup\{u_i\}$. And, since $\Theta$ is a congruence with respect to
a semilattice operation, we immediately obtain
\begin{equation*}
\bigcup_{i=1}^{k}U_i\mathrel{\Theta}\bigcup_{i=1}^{k}U_i\cup
U\mathrel{\Theta} \bigcup_{i=1}^{k}U_i\cup
U\cup\{a\}\mathrel{\Theta} \bigcup_{i=1}^{k}U_i\cup\{a\}.
\end{equation*}
This means that $a\in \cc_{\Theta}(T)$ and $\cc_{\Theta}$ is an $\varnothing$-preserving
closure operator. Moreover, the condition $\cc_{\Theta}(T)=
\bigcup\{\cc_{\Theta}(U)\mid U \; {\rm is \; a \;finite\; subset \;
of} \; T\}$ is clearly satisfied. Hence, the operator $\cc_{\Theta}$
is an algebraic $\varnothing$-preserving closure operator.
\end{proof}

\begin{lm}\label{sec5:thm2}
Let $(A,\Om)$ be an algebra and let $\Theta$ be a congruence
relation on $(\wpf A,\Om,\cup)$. Then, for each $0<n$-ary operation
$\om\in \Om$, the algebraic $\varnothing$-preserving closure operator $\cc_{\Theta}$
satisfies the condition:
\begin{equation}\label{sec5:eq1}
\om(\cc_{\Theta}(T_1),\ldots,\cc_{\Theta}(T_n))\subseteq
\cc_{\Theta}(\om(T_1,\ldots,T_{n})), \; \text{for} \;\;
T_1,\ldots,T_n\in \wpf A.
\end{equation}
\end{lm}

\begin{proof}
Let $T_1,\ldots,T_n\in \wpf A$ and consider an element $x$ in
$\om(\cc(T_1),\ldots,\cc(T_n))$. This means $x=\om(a_1,\ldots,a_n)$
for some $a_i\in\cc(T_i)$. Then for each $i=1,\ldots ,n$ there
exists a finite subset $U_i$ of $T_i$ such that $U_i\mathrel{\Theta}
U_i\cup\{a_i\}$. Since $\Theta$ is a congruence on $(\wpf
A,\Om,\cup)$, one obtains that
$\om(U_1\cup\{a_1\},\ldots,U_n\cup\{a_n\})\mathrel{\Theta}\om(U_1,\ldots,U_n)$.
Now by \eqref{consub}
$\om(U_1,\ldots,U_n)\cup\{\om(a_1,\ldots,a_n)\}\subseteq
\om(U_1\cup\{a_1\},\ldots,U_n\cup\{a_n\})$. Hence, one has
\begin{equation*}
\om(U_1,\ldots,U_n)\mathrel{\Theta}\om(U_1\cup\{a_1\},\ldots,U_n\cup\{a_n\})\mathrel{\Theta}
\om(U_1,\ldots,U_n)\cup\{\om(a_1,\ldots,a_n)\}.
\end{equation*}
Since $\om(U_1,\ldots,U_n)\subseteq\om(T_1,\ldots,T_n)$, this means
that $x\in\cc_{\Theta}(\om(T_1,\ldots,T_{n}))$ and finishes the
proof.
\end{proof}
For an algebra $(A,\Om)$ denote by $\Con(\wpf A)$ the set of all
congruence relations on $(\wpf A,\Om,\cup)$ and by $\Clo(A)$ the set
of all algebraic $\varnothing$-preserving closure operators on $A$ which satisfy
condition \eqref{sec5:eq1} for each $0<n$-ary operation $\om\in
\Om$.

\begin{lm}\label{sec5:thm1}
Let $(A,\Om)$ be an algebra and let $\cc\in \Clo(A)$. Then the
relation $\Upsilon_{\cc}$ is a congruence on $(\wpf A,\Om,\cup)$.
\end{lm}

\begin{proof}
Let $\cc\in \Clo(A)$ and let $A_1,\ldots,A_n,B_1,\ldots,B_n\in \wpf A$ be such that
$\cc(A_1)=\cc(B_1),\ldots,\cc(A_n)=\cc(B_n)$.  In Example \ref{sec3:exm4} we showed that
\begin{align*}
\cc(\omega(A_1,\ldots,A_n))&=\cc(\om(\cc(A_1),\ldots,\cc(A_n)))\;\;\rm{and}\\
\cc(A_1\cup A_2)&=\cc(\cc(A_1)\cup \cc(A_2)).
\end{align*}
This immediately implies that $\omega(A_1,\ldots,A_n)\;\Upsilon_{\cc}\;\om(B_1,\ldots,B_{n})$ and
$A_1\cup A_2\;\Upsilon_{\cc}\;B_1\cup B_2$.
\end{proof}


\begin{lm}\label{cordod} Let $(A,\Om)$ be an algebra and let $\Theta\in \Con(\wpf A)$. For $Q,R\in\wpf A$
\begin{equation*}
Q\subseteq \cc_{\Theta}(R)\;\;\Leftrightarrow\;\;Q\cup
R\mathrel{\Theta} R.
\end{equation*}
\end{lm}
\begin{proof}
Let $Q,R\in\wpf A$ and $Q\subseteq \cc_{\Theta}(R)$. This means that
for each $q\in Q$, there exists a subset $S_q\subseteq R$ such that
$S_q\cup\{q\}\mathrel{\Theta}S_q$. This implies $
R\cup\{q\}\mathrel{\Theta}R$ for each $q\in Q$. Hence, $R\cup
Q\mathrel{\Theta}R$.

Since $R\subseteq \{q\}\cup R\subseteq Q\cup R$ and the congruence classes of semilattice congruences are convex, we obtain that
$R\mathrel{\Theta} R\cup \{q\}$. Hence, $q\in \cc_{\Theta}(R)$.
\end{proof}
\begin{thm}\label{sec5:lem2} Let $(A,\Om)$ be an algebra. There is a
one-to-one correspondence between the sets $\Con(\wpf A)$ and
$\Clo(A)$. Namely, for $\Theta\in \Con(\wpf A)$ and $\cc\in \Clo(A)$,
one has
\begin{equation*}
\Theta=\Upsilon_{\cc_{\Theta}}\;\;\text{and}\;\;\cc=\cc_{\Upsilon_{\cc}}.
\end{equation*}
\end{thm}

\begin{proof}
Let $Q,R\in \wpf A$ and let $Q\mathrel{\Theta}R$. 
We have that $Q\cup R\mathrel{\Theta} R\cup R=R$. Thus, $Q\subseteq
\cc_{\Theta}(R)$ by Lemma \ref{cordod}. By Lemma \ref{sec5:lem10}, $\cc_{\Theta}$ is a
closure operator, so we have $\cc_{\Theta}(Q)\subseteq
\cc_{\Theta}(R)$. Similarly, we obtain the inclusion
$\cc_{\Theta}(Q)\supseteq \cc_{\Theta}(R)$. Therefore,
$Q\mathrel{\Upsilon_{\cc_{\Theta}}} R$, and $\Theta
\subseteq\Upsilon_{\cc_{\Theta}}$.

On the other hand, let $Q\mathrel{\Upsilon_{\cc_{\Theta}}} R$. By
Lemma \ref{cordod} we obtain that $R\cup Q\mathrel{\Theta}R$ and
$R\cup Q\mathrel{\Theta}Q$. So $R\mathrel{\Theta}Q$, and
consequently, $\Upsilon_{\cc_{\Theta}}\subseteq \Theta$.
%

The rest follows by easy  observation that for an algebraic closure
operator on a set $A$ and any subset $T\subseteq A$ one has
\begin{align*}
&\cc(T)=\bigcup\{\cc(U)\mid U\in \wpf T \}=\bigcup\limits_{U\in \wpf
T}\{a\in A\mid a\in
\cc(U)\}=\\
&\bigcup\limits_{U\in \wpf T}\{a\in A\mid  \cc(U\cup\{a\})=\cc(U) \}
=\{a\in A\mid \exists (U\in \wpf{T})\;\; U\cup
\{a\}\mathrel{\Upsilon_{\cc}} U\}.
\end{align*}
\end{proof}

Now we focus on closure operators from the set $\Clo(A)$
corresponding to fully invariant congruences on $(\wpf A,
\Om,\cup)$. By Lemma \ref{cordod} one immediately obtains
\begin{lm}\label{sec6:lem2}
Let $(A,\Om)$ be an algebra and $\Theta$ be a fully invariant
congruence on $(\wpf A,\Om,\cup)$. Then, for each endomorphism
$\varphi$ of $(\wpf A,\Om,\cup)$, the operator $\cc_{\Theta}$
satisfies the following condition:
\begin{equation}\label{sec6:eq1}
\varphi(\{r\})\subseteq \cc_{\Theta}(\varphi(T)), \;\;for \;\; T\in
\wpf A\;\; and \;\; r\in \cc_{\Theta}(T).
\end{equation}
\end{lm}


\begin{exm}\label{sec6:exm1}
Let $(A,\Om)$ be a mode (idempotent and entropic algebra). Let
us define on the set $\mathcal{P}^{<\omega}_{>0}A$ the following
relation: for $X,Y\in \mathcal{P}^{<\omega}_{>0}A$
\begin{equation*}
X \mathrel{\varrho} Y \;\; \Leftrightarrow \;\;\langle X\rangle=\langle
Y\rangle,
\end{equation*}
where $\langle X\rangle$ denotes the subalgebra of $(A,\Om)$
generated by $X$.

By results of \cite{PZ12}, the relation $\varrho$ is a fully
invariant congruence on $(\mathcal{P}^{<\omega}_{>0}A,\Om,\cup)$.
This implies that the mapping $\cc_{\varrho}:\mathcal{P}A\to
\mathcal{P}A$, $T\mapsto \langle T\rangle$, is an algebraic $\varnothing$-preserving
closure operator on $\mathcal{P}^{<\omega}_{>0}A$, which satisfies
the conditions \eqref{sec5:eq1} and \eqref{sec6:eq1}.
\end{exm}


\begin{exm} It was shown in \cite{B93} that if $f\colon A\to B$ is
a homomorphism of $\Om$-algebras $(A,\Om)$ and $(B,\Om)$, then so is
$f^{+}\colon\wpf A\to \wpf B$, where $f^{+}(S):=\{f(s)\mid s\in S\}$
for any $S\in\wpf A$. Obviously, $f^{+}$ is also a homomorphism with
respect to $\cup$. If $f$ is an endomorphism of the algebra
$(A,\Om)$, then $f^{+}$ is an endomorphism of the algebra $(\wpf
A,\Om,\cup)$. Let $\Theta$ be a fully invariant congruence on $(\wpf
A,\Om,\cup)$. Then, for each endomorphism $f$ of $(A,\Om)$, the
operator $\cc_{\Theta}$ satisfies the following condition:
\begin{equation*}\label{sec6:eq11}
f^{+}(\cc_{\Theta}(T))\subseteq \cc_{\Theta}(f^{+}(T)), \;\;for \;\;
T\in \wpf A.
\end{equation*}
\end{exm}
\vskip 5mm For an algebra $(A,\Om)$ let us denote by $\Con_{\fic}(\wpf
A)$ the set of all fully invariant
congruence relations in $\Con(\wpf A)$ and by $\Clo_{\fic}(A)$  
the set of all closure operators in $\Clo(A)$ which satisfy condition
\eqref{sec6:eq1}. By Lemma \ref{sec5:thm1}
one immediately obtains
\begin{lm}\label{sec6:lem3}
Let $(A,\Om)$ be an algebra and $\cc\in \Clo_{\fic}(A)$. Then the
relation $\Upsilon_{\cc}$ is a fully invariant congruence on $(\wpf
A,\Om,\cup)$.
\end{lm}

\begin{exm}\cite{PZ12}\label{sec6:exm2}
Let $\mathcal{M}$ be a variety of $\Om$-modes and
$(F_{\mathcal{M}}(X),\Om)$ be the $\mathcal{M}$-free mode over a set
$X$. Let $\Gamma\subseteq\Om$ and let $\cc_{\Gamma}\colon\p
F_{\mathcal{M}}(X)\to \p F_{\mathcal{M}}(X)$ be the closure operator
described in Example \ref{sec3:exm5}. As we have already noticed in
\cite{PZ12}, for each $n$-ary complex operation $\om\in \Om$, and
any non-empty subsets $X_1,\ldots,X_n$ of $F_{\mathcal{M}}(X)$,
\begin{equation*}
\cc_{\Gamma}(\om(X_1,\ldots,X_n)) =
\om(\cc_{\Gamma}(X_1),\ldots,\cc_{\Gamma}(X_n)).
\end{equation*}

Moreover, we proved in \cite{PZ12} that for any $T\in \wpf
F_{\mathcal{M}}(X)$ and any endomorphism $\varphi$ of $(\wpf
F_{\mathcal{M}}(X),\Om,\cup)$, if $r\in \cc_{\Gamma}(T)$ then we
have $\varphi(\{r\})\subseteq \cc_{\Gamma}(\varphi(T))$. Hence,
$\cc_{\Gamma}$ is an algebraic $\varnothing$-preserving closure operator which satisfies
conditions \eqref{sec5:eq1} and \eqref{sec6:eq1}. Therefore, the
relation $\Theta_{\cc_{\Gamma}}\subseteq\wpf
F_{\mathcal{M}}(X)\times \wpf F_{\mathcal{M}}(X)$ such that
\begin{equation*}
(Q,R)\in \Theta_{\cc_{\Gamma}} \;\; \Leftrightarrow\;\; \langle
Q\rangle_{\Gamma}=\langle R\rangle_{\Gamma},
\end{equation*}
is a fully invariant congruence relation on $(\wpf
F_{\mathcal{M}}(X),\Om,\cup)$.
\end{exm}

By Theorem \ref{sec5:lem2}, Lemma \ref{sec6:lem2} and Lemma
\ref{sec6:lem3} we have
\begin{thm}\label{sec4:cor2}
Let $(A,\Om)$ be an algebra. There is a one-to-one correspondence
between the sets $\Con_{\fic}(\wpf A)$ and $\Clo_{\fic}(A)$.
\end{thm}

Let us define on $\Clo_{\fic}(A)$
a binary relation $\leq$ in the following way: for any
$\cc_1,\cc_2\in \Clo_{\fic}(A)$
\begin{equation*}
\cc_1\leq\cc_2\;\Leftrightarrow\;\Upsilon_ {\cc_2}\subseteq\Upsilon_
{\cc_1}.
\end{equation*}
Obviously, $(\Clo_{\fic}(A),\leq)$
is a complete lattice. \vskip5mm In Section 5 we use some
special operators from the set $\Clo_{\fic}(A)$ to study particular
subvarieties of semilattice ordered algebras. To do it, we need the
following sequence of results; their proofs are straightforward.

Let $A$ be a set and $\Theta$ be an equivalence relation on $\wpf
A$. We can define a binary relation $\widetilde{\Theta}\subseteq
A\times A$ in the following way: for $a,b\in A$
\begin{equation*}
(a,b)\in \widetilde{\Theta} \;\; \Leftrightarrow\;\;
(\{a\},\{b\})\in \Theta.
\end{equation*}
Clearly, $\widetilde{\Theta}$ is an equivalence relation.
Additionally, if $(A,\Om)$ is an algebra and $\Theta$ is a
congruence on $(\wpf A,\Om,\cup)$, then also  $\widetilde{\Theta}$
is a  congruence relation on $(A,\Om)$.


\begin{lm}\label{sec5:lem9}
Let $I$ be a set and for each $i\in I$, let $\Theta_i$ be an
equivalence relation on $\wpf A$. Then
\begin{equation*}
\bigcap\limits_{i\in
I}\widetilde{\Theta}_i=\widetilde{\bigcap\limits_{i\in I}\Theta_i}.
\end{equation*}
\end{lm}
For a set $Q\in \wpf A$ and a congruence $\alpha$ on $(A,\Om)$, let
$Q^{\alpha}$ denote the set $\{q/\alpha\mid q\in Q\}$. Let $\Psi\in
\Con(\wpf A)$ be such that $\widetilde{\Psi}=\alpha$. We define a
relation $\delta_{\Psi}\subseteq\wpf (A/\alpha)\times
\wpf(A/\alpha)$ in the following way:
\begin{equation}\label{rel_delta}
(B^\alpha,C^\alpha)\in \delta_{\Psi}\;\; \Leftrightarrow\;\;
(B,C)\in \Psi
\end{equation}
for $B,C\in \wpf A$. Conversely, let $\psi\in \Con(\wpf(A/\alpha))$
be such that $\widetilde{\psi}=\mathbf{0}_{A/\alpha}$. Define a relation
$\Delta_{\psi}\subseteq\wpf A\times \wpf A$:
\begin{equation}\label{rel_D}
(B,C)\in\Delta_{\psi}\;\; \Leftrightarrow\;\; (B^\alpha,C^\alpha)\in
\psi
\end{equation}
for $B,C\in \wpf A$.
Straightforward calculations show that the definitions of both relations are correct. Moreover,  $\delta_{\Psi}\in
\Con(\wpf(A/\alpha))$ and $\Delta_{\psi}\in \Con(\wpf A)$.
\begin{thm}\label{sec5:lem8}
Let $(A,\Om)$ be an algebra, $\alpha$ be a congruence on $(A,\Om)$,
$\Theta$ be a congruence relation on $(\wpf A,\Om,\cup)$ with
$\widetilde{\Theta}=\alpha$, and $\psi$ be a congruence on
$(\wpf(A/\alpha),\Om,\cup)$ such that
$\widetilde{\psi}=\mathbf{0}_{A/\alpha}$. Then
\begin{equation}\label{sec4:eq1}
\Theta=\Delta_{\delta_{\Theta}}\;\; and
\end{equation}
\begin{equation}\label{sec4:eq2}
\psi=\delta_{\Delta_{\psi}}.
\end{equation}
\end{thm}

\begin{lm}\label{sec4:lem3} Let $A$ be a set, $\Theta$ be an
equivalence relation on $\wpf A$ such that
$\widetilde{\Theta}=\mathbf{0}_A$. Then, the closure operator $\cc_{\Theta}$
satisfies the following condition: for $a,b\in A$
\begin{equation}\label{sec5:eq5}
\cc_{\Theta}(\{a\})=\cc_{\Theta}(\{b\})  \;\; \Leftrightarrow\;\;
a=b.
\end{equation}
Conversely, let $\cc\colon\p A\to\p A$ be a closure operator on $A$
satisfying the condition \eqref{sec5:eq5}, then
$\widetilde{\Upsilon}_{\cc}=\mathbf{0}_A$.
\end{lm}

By Theorem \ref{sec4:cor2} and Lemma \ref{sec4:lem3} one immediately
has

\begin{cor}\label{sec5:cor2} For an algebra $(A,\Om)$ there is a
one-to-one correspondence between the sets $\{\Theta\in
\Con_{\fic}(\wpf A)\;|\;\widetilde{\Theta}=\mathbf{0}_A\}$ and $\{\cc\in
\Clo_{\fic}(A)\;|\;\cc\;\;\text{satisfies}\;\; \eqref{sec5:eq5}\}$.
\end{cor}

Finally, Theorems \ref{sec5:lem2} and \ref{sec5:lem8} imply

\begin{cor}\label{sec5:cor3} Let $(A,\Om)$ be an algebra, and
$\Theta$ be a congruence on $(\wpf A,\Om,\cup)$. There is a one-to-one
correspondence between the following sets:
\begin{itemize}
\item[(i)] $\{\Psi\in
\Con(\wpf A)\;|\;\widetilde{\Psi}=\widetilde{\Theta}\}$;
\item[(ii)] $\{\psi\in
\Con(\wpf
A/\widetilde{\Theta})\;|\;\widetilde{\psi}=\mathbf{0}_{A/\widetilde{\Theta}}\}$;
\item[(iii)] $\{\cc\in
\Clo_{\fic}(A/\widetilde{\Theta})\;|\;\cc\;\;\text{satisfies}\;\;
\eqref{sec5:eq5}\}$.
\end{itemize}
\end{cor}
\section{Applications: the lattice of subvarieties}\label{sec5}

Let $\vv$ be a subvariety of $\typ$ and let $\mm_{\vv}$ denote the
variety of all semilattice ordered $\vv$-algebras. In \cite{PZ12a}
it was shown that the extended power algebra $(\wpf
F_{\vv}(X),\Om,\cup)$ defined on the free algebra $(F_{\vv}(X),\Om)$
over a set $X$ in the variety $\vv$ has the universality property
for all semilattice ordered $\vv$-algebras. In particular,

\begin{thm}\label{sec4:cor1}\cite{PZ12a}
The semilattice ordered algebra $(\wpf F_{\vv}(X),\Om,\cup)$ is free
over a set $X$ in the variety $\mm_{\vv}$ if and only if $(\wpf
F_{\vv}(X),\Om,\cup)\in\mm_{\vv}$.
\end{thm}

\begin{cor}\label{sec4:thm2}\cite{PZ12a}
Let $(F_{\typ}(X),\Om)$ be the free algebra over a set $X$ in the variety $\typ$.
The extended power algebra $(\wpf F_{\typ}(X),\Om,\cup)$ is free
over $X$ in the variety $\mm_{\typ}$ of all semilattice
ordered $\typ$-algebras.
\end{cor}

Note that, by results of G. Gr\"atzer and H. Lakser \cite{GL88}, the
same holds also for any \emph{linear variety} - a variety defined by
identities of the form $t\=u$ where every variable occurs at each
side at most once, i.e. by \emph{linear} identities.

\begin{thm}\label{sec4:thm3}
Let $\vv$ be a linear variety and let $(F_{\vv}(X),\Om)$ be the free
algebra over a set $X$ in $\vv$. The extended power algebra $(\wpf
F_{\vv}(X),\Om,\cup)$ is free over $X$ in the variety
$\mm_{\vv}$ of all semilattice ordered $\vv$-algebras.
\end{thm}

For a variety $\vv$ let $\vv^{*}$ be its \emph{linearization}, the
variety defined by all linear identities satisfied in $\vv$.
Obviously, $\vv^{*}$ contains $\vv$ as a subvariety.

G. Gr\"atzer and H. Lakser proved in \cite[Prop. 1]{GL88} that for any subvariety
$\vv\subseteq\mho$ the algebra $(\wpf F_{\vv}(X),\Om)$ satisfies only those identities
resulting through identification of variables from the linear identities true in $\vv$. This
implies that for each subvariety $\vv\subseteq\mho$, the algebra $(\wpf F_{\vv}(X),\Om)$ belongs to $\vv^*$,
but it does not belong to any its proper subvariety. Hence, by Corollary \ref{sec4:cor1} and Theorem \ref{sec4:thm3} we obtain

\begin{thm}\label{sec4:thm4}
The semilattice ordered algebra
$(\wpf F_{\vv}(X),\Om,\cup)$ is free over a set $X$ in the variety
$\mm_{\vv}$ if and only if $\vv=\vv^*$.
\end{thm}

\begin{cor}\label{sec4:cor6} A variety $\vv$ is linear if and only
if $(\wpf F_{\vv}(X),\Om)\in\vv$.
\end{cor}

Theorem \ref{sec4:thm4}, Corollary \ref{sec4:cor6}, and the following theorem give a new
characterization of the linear varieties - the semantic one instead
of well-established syntactic one.

\begin{thm}\label{varlin}
A variety $\vv$ is linear if and only if it is a variety closed
under power construction $A\mapsto \wpf A$.
\end{thm}

By Corollary \ref{sec4:cor2}, for a subvariety $\vv\subseteq\typ$,
we have a one-to-one correspondence between the set $\Con_{\fic}(\wpf
F_{\vv}(X))$ of all fully invariant congruence relations on $(\wpf
F_{\vv}(X),\Om,\cup)$ and the set $\Clo_{\fic}(F_{\vv}(X))$ 
of all algebraic $\varnothing$-preserving closure operators on $F_{\vv}(X)$ which
satisfy conditions \eqref{sec5:eq1} and \eqref{sec6:eq1}.

From now we assume that $X$ is an infinite set. By Corollary
\ref{sec4:cor1} if $(\wpf F_{\vv}(X),\Om,\cup)\in \mm_{\vv}$ then
there is a one-to-one correspondence between the set of all
subvarieties of the variety $\mm_{\vv}={\rm HSP}((\wpf
F_{\vv}(X),\Om,\cup))$ and the set $\Clo_{\fic}(F_{\vv}(X))$.

In particular, by Theorems \ref{sec4:thm2} and \ref{sec4:thm3}, we
obtain

\begin{cor}
There is a one-to-one correspondence between the set of all
subvarieties of the variety $\mm_{\typ}={\rm HSP}((\wpf
F_{\typ}(X),\Om,\cup))$ and the set $\Clo_{\fic}(F_{\typ}(X))$.
\end{cor}

\begin{cor}
Let $\vv$ be a linear variety. There is a one-to-one correspondence
between the set of all subvarieties of the variety $\mm_{\vv}={\rm
HSP}((\wpf F_{\vv}(X),\Om,\cup))$ and the set
$\Clo_{\fic}(F_{\vv}(X))$.
\end{cor}

Recall that for any $\cc_1,\cc_2\in \Clo_{\fic}(F_{\vv}(X))$
\begin{equation*}
\cc_1\leq\cc_2\;\Leftrightarrow\;\Upsilon_ {\cc_2}\subseteq\Upsilon_
{\cc_1}.
\end{equation*}

\begin{thm}\label{sec6:thm1} For any subvariety $\vv$ of $\typ$,
the lattice of all subvarieties of the variety ${\rm HSP}((\wpf
F_{\vv}(X),\Om,\cup))$ is isomorphic to the lattice
$(\Clo_{\fic}(F_{\vv}(X)),\leq)$.
\end{thm}

Let $\mathcal{L}(\mm_{\vv})$ denote the set of all subvarieties of
the variety $\mm_{\vv}$. By Theorem \ref{sec4:thm3} we immediately
obtain the following result.

\begin{thm}\label{sec6:thm2}
Let $\vv\subseteq\typ$ be a variety defined by linear identities.
The lattice $(\mathcal{L}(\mm_{\vv}),\subseteq)$ 
is isomorphic to the lattice $(\Clo_{\fic}(F_{\vv}(X)),\leq)$.
\end{thm}

\begin{cor}\label{sec6:cor2} The lattice $(\mathcal{L}(\mm_{\typ}),\subseteq)$
of all subvarieties of the variety $\mm_{\typ}$ is isomorphic to the
lattice $(\Clo_{\fic}(F_{\typ}(X)),\leq)$.
\end{cor}

\section{Applications: $\vv$-preserved subvarieties}\label{sec4}

In \cite{PZ12a} we showed that for any subvariety $\vv$ of $\typ$
and $\Theta\in \Con_{\fic}(\wpf F_{\vv}(X))$, the relation
$\widetilde{\Theta}$ is a fully invariant congruence on
$(F_{\vv}(X),\Om)$. This means that each fully invariant congruence
relation $\Theta$ on $(\wpf F_{\vv}(X),\Om,\cup)$ determines a
subvariety of $\vv$. In particular, each fully invariant congruence
relation $\Theta$ on $(\wpf F_{\typ}(X),\Om,\cup)$ determines a
subvariety $\typ_{\widetilde{\Theta}}$ of $\typ$:
\begin{equation*}
\typ_{\widetilde{\Theta}}:={\rm
HSP}((F_{\typ}(X)/\widetilde{\Theta},\Om)).
\end{equation*}
Moreover, by results of \cite{PZ12a}, for each subvariety
$\mm_{\vv}\subseteq\mm_{\typ}$, with $\vv\subseteq\typ$, there
exists a congruence $\Theta\in \Con_{\fic}(\wpf F_{\typ}(X))$ such that
$\mm_{\vv}=\mm_{\typ_{\widetilde{\Theta}}}$. Recall we assume $X$ is
an infinite set.

\begin{de}\cite{PZ12a} Let $\vv$ be a subvariety of $\typ$ and
$\mathcal{K}$ be a non-trivial subvariety of $\mm_{\vv}$. The variety $\mathcal{K}$
is $\vv$-\emph{preserved}, if $\mathcal{K}\nsubseteq\mm_{\mathcal{W}}$ for
any proper subvariety $\mathcal{W}$ of $\vv$.
\end{de}
The following characterization of $\vv$-preserved subvarieties of
the variety $\mm_{\vv}$ was presented in \cite{PZ12a}.
\begin{lm}\cite{PZ12a}\label{sec6:lem9}
Let $\Theta,\Psi\in \Con_{\fic}(\wpf F_{\typ}(X))$. A non-trivial
subvariety
\begin{equation*}
\mathcal{K}={\rm HSP}((\wpf
F_{\typ}(X)/\Psi,\Om,\cup))\subseteq\mm_{\typ_{\widetilde{\Theta}}}
\end{equation*}
is $\typ_{\widetilde{\Theta}}$-preserved if and only if
$\widetilde{\Psi}=\widetilde{\Theta}$.
\end{lm}

In \cite{PZ12a} we showed that for each congruence $\Theta\in \Con_{\fic}(\wpf
F_{\typ}(X))$ there is a correspondence between the set of $\typ_{\widetilde{\Theta}}$-preserved
subvarieties of $\mm_{\typ_{\widetilde{\Theta}}}$ and the set of some fully invariant congruence relations on the algebra
$(\wpf (F_{\typ}(X)/\widetilde{\Theta}),\Om,\cup)$.

We introduce the following notation. For a fully invariant
congruence $\alpha\in \Con_{\fic}(F_{\typ}(X))$, let $\typ_{\alpha}:=
{\rm HSP}((F_{\typ}(X)/\alpha,\Om))\subseteq\typ$ and
$(F_{\typ_{\alpha}}(X),\Om)=(F_{\typ}(X)/\alpha,\Om)$ be the free
algebra in $\typ_{\alpha}$. Denote by $\Con_{\fic}^{\mathbf{0}}(\wpf
F_{\typ_{\alpha}}(X))$ the following set:
\begin{align*}
\{\psi\in \Con_{\fic}(\wpf
F_{\typ_{\alpha}}(X))\mid \widetilde{\psi}=\mathbf{0}_{F_{\typ_{\alpha}}(X)}\; {\rm and }\;
(\wpf F_{\typ_{\alpha}}(X)/{\psi},\Om)\in \typ_{\alpha}\}.
\end{align*}
Recall the relations $\delta_{\Psi}$ and $\Delta_{\psi}$ defined
by (\ref{rel_delta}) and (\ref{rel_D}), respectively. The following
technical lemma was proved in \cite{PZ12a}.

\begin{lm}\cite{PZ12a}\label{sec5:lem6} Let $\Psi\in
\Con_{\fic}(\wpf F_{\typ}(X))$. Then $\delta_{\Psi}\in
\Con_{\fic}^{\mathbf{0}}(\wpf F_{\typ_{\widetilde{\Psi}}}(X))$. Conversely, let
$\alpha\in \Con_{\fic}(F_{\typ}(X))$ and $\psi\in \Con_{\fic}^{\mathbf{0}}(\wpf
F_{\typ_\alpha}(X))$. Then, $\Delta_{\psi}\in \Con_{\fic}(\wpf
F_{\typ}(X))$, and moreover ${\rm HSP}((\wpf
F_{\typ}(X)/\Delta_{\psi},\Om))\subseteq \typ_{\alpha}$ and
$\widetilde{\Delta}_{\psi}=\alpha$.
\end{lm}

Corollary \ref{sec5:cor3} and Lemma \ref{sec5:lem6} immediately
imply
\begin{cor}\cite{PZ12a}\label{sec5:cor1}
For each $\Theta\in \Con_{\fic}(\wpf F_{\typ}(X))$
there is a one-to-one correspondence between the set of
all $\typ_{\widetilde{\Theta}}$-preserved
subvarieties of $\mm_{\typ_{\widetilde{\Theta}}}$ and the set $\Con_{\fic}^{\mathbf{0}}(\wpf F_{\typ_{\widetilde{\Theta}}}(X))$.
\end{cor}

Now we will show that there is a similar correspondence between the
set of all $\typ_{\widetilde{\Theta}}$-preserved subvarieties of
$\mm_{\typ_{\widetilde{\Theta}}}$ and the set of all algebraic $\varnothing$-preserving
closure operators on $F_{\typ_{\widetilde{\Theta}}}(X)$ which
satisfy the conditions \eqref{sec5:eq1}, \eqref{sec6:eq1},
\eqref{sec5:eq5} and additional one.

Recall that for a set $Q\subseteq A$ and an equivalence relation $\alpha\subseteq A\times A$,
$Q^{\alpha}$ denotes the set $\{q/\alpha\mid q\in Q\}$.

\begin{lm}\label{sec4:lem2}
Let $\alpha\in \Con_{\fic}(F_{\typ}(X))$. Let $\cc$ be a closure
operator on $F_{\typ_{\alpha}}(X)$ which satisfies
\eqref{sec5:eq1} 
and the following condition: for $\{s\},Q,P_1,\ldots,P_n\in \wpf
F_{\typ}(X)$
\begin{align}\label{sec4:eq4}
&s/\alpha\in \cc(Q^\alpha)\;\;\Rightarrow\\
&s(P_1^{\alpha},\ldots,P_n^{\alpha})\subseteq
\cc(\{q(P_1^{\alpha},\ldots,P_n^{\alpha})\mid q\in Q\}).\nonumber
\end{align}
Moreover, let $\Upsilon_{\cc}\subseteq\wpf
F_{\typ_\alpha}(X)\times\wpf F_{\typ_\alpha}(X)$ be the congruence
associated with $\cc$. Then, $(\wpf
F_{\typ_\alpha}(X)/\Upsilon_{\cc},\Om)\in \typ_{\alpha}$.
\end{lm}
\begin{proof}
Let $\cc$ be a a closure operator on $F_{\typ_{\alpha}}(X)$ which
satisfies \eqref{sec5:eq1} and \eqref{sec4:eq4}. Let $t,u\in
F_{\typ}(X)$ and $(t,u)\in\alpha$. Then $t/\alpha=u/\alpha$ and
obviously, $t/\alpha\in \cc(\{u/\alpha\})$.

Hence, by \eqref{sec4:eq4}, for any $P_1,\ldots,P_n\in \wpf F_{\typ}(X)$
\begin{align*}
t(P_1^{\alpha},\ldots,P_n^{\alpha})\subseteq \cc(u(P_1^{\alpha},\ldots,P_n^{\alpha})).
\end{align*}
This implies
\begin{align*}
\cc(t(P_1^{\alpha},\ldots,P_n^{\alpha}))\subseteq \cc(u(P_1^{\alpha},\ldots,P_n^{\alpha})).
\end{align*}
Similarly, we show that $\cc(u(P_1^{\alpha},\ldots,P_n^{\alpha}))\subseteq \cc(t(P_1^{\alpha},\ldots,P_n^{\alpha}))$, and consequently we obtain $\cc(t(P_1^{\alpha},\ldots,P_n^{\alpha}))=\cc(u(P_1^{\alpha},\ldots,P_n^{\alpha}))$.\\
Since $\Upsilon_{\cc}$ is the congruence associated with $\cc$, for any $P_1,\ldots,P_n\in \wpf F_{\typ}(X)$, we have
\begin{align*}
&\cc(t(P_1^{\alpha},\ldots,P_n^{\alpha}))=\cc(u(P_1^{\alpha},\ldots,P_n^{\alpha}))\;\;\Leftrightarrow\\
&(t(P_1^{\alpha},\ldots,P_n^{\alpha}),u(P_1^{\alpha},\ldots,P_n^{\alpha}))\in \Upsilon_{\cc}\;\;\Leftrightarrow\\
&t(P_1^{\alpha},\ldots,P_n^{\alpha})/\Upsilon_{\cc}=u(P_1^{\alpha},\ldots,P_n^{\alpha})/\Upsilon_{\cc}\;\;\Leftrightarrow\\
&t(P_1^{\alpha}/\Upsilon_{\cc},\ldots,P_n^{\alpha}/\Upsilon_{\cc})=
u(P_1^{\alpha}/\Upsilon_{\cc},\ldots,P_n^{\alpha}/\Upsilon_{\cc}),
\end{align*}
which proves that each identity $t\approx u$ true in the variety
$\typ_{\alpha}$ is also satisfied in $(\wpf
F_{\typ_\alpha}(X)/\Upsilon_{\cc},\Om)$ and completes the proof.
\end{proof}

Let $\alpha\in \Con_{\fic}(F_{\typ}(X))$. Let us denote by
$\Clo_{\fic}^{\mathbf{0}}(F_{\typ_{\alpha}}(X))$ the set of all algebraic $\varnothing$-preserving
closure operators on $F_{\typ_{\alpha}}(X)$ which satisfy the
conditions: \eqref{sec5:eq1}, \eqref{sec6:eq1}, \eqref{sec5:eq5} and
\eqref{sec4:eq4}.

By Lemmas \ref{sec6:lem3}, \ref{sec4:lem3}, \ref{sec4:lem2} and
\ref{sec5:lem6}, one obtains the following.
\begin{cor}\label{sec4:cor3}
Let $\alpha\in \Con_{\fic}(F_{\typ}(X))$ and $\cc\in
\Clo_{\fic}^{\mathbf{0}}(F_{\typ_{\alpha}}(X))$.
Moreover, let $\Upsilon_{\cc}\in \Con(\wpf F_{\typ_\alpha}(X))$ be
the congruence associated with $\cc$. Then, the relation
$\Delta_{\Upsilon_{\cc}}$ is a fully invariant congruence on $(\wpf
F_{\typ}(X),\Om,\cup)$ such that $(\wpf
F_{\typ}(X)/\Delta_{\Upsilon_{\cc}},\Om,)\in \typ_{\alpha}$ and
$\widetilde{\Delta}_{\Upsilon_{\cc}}=\alpha$.
\end{cor}

\begin{lm}\label{sec4:lem10}
Let $\Theta\in \Con_{\fic}(\wpf F_{\typ}(X))$. Then, 
$\cc_{\delta_{\Theta}}\in \Clo_{\fic}^{\mathbf{0}}(F_{\typ_{\widetilde{\Theta}}}(X))$.
\end{lm}
\begin{proof}

By Lemmas \ref{sec5:lem6} and \ref{sec4:lem3} one immediately
obtains that $\cc_{\delta_{\Theta}}$ is an algebraic $\varnothing$-preserving closure
operator on $F_{\typ_{\widetilde{\Theta}}}(X)$ which satisfies the
conditions \eqref{sec5:eq1}, \eqref{sec6:eq1} and \eqref{sec5:eq5}.
Now we will show that the operator $\cc_{\delta_{\Theta}}$ satisfies
also the condition \eqref{sec4:eq4}.

Let $\Theta\in \Con_{\fic}(\wpf F_{\typ}(X))$. By Lemma
\ref{sec5:lem6}, we have $\delta_{\Theta}\in \Con_{\fic}^{\mathbf{0}}(\wpf
F_{\typ_{\widetilde{\Theta}}}(X))$. Now let $\{s\},Q\in\wpf
F_{\typ}(X)$ and $s/\widetilde{\Theta}\in
\cc_{\delta_{\Theta}}(Q^{\widetilde{\Theta}})$. Since
$\cc_{\delta_{\Theta}}$ is the closure operator determined by the
congruence $\delta_{\Theta}$,
\begin{equation*}
(\{s\}\cup Q)^{\widetilde{\Theta}}\mathrel{\delta_{\Theta}}
Q^{\widetilde{\Theta}}.
\end{equation*}
By Theorem \ref{sec5:lem8} one obtains
\begin{align*}
(\{s\}\cup Q)^{\widetilde{\Theta}}\mathrel{\delta_{\Theta}}
Q^{\widetilde{\Theta}}\mathrel{\Leftrightarrow}\{s\}\cup Q
\mathrel{\Theta} Q \mathrel{\Leftrightarrow} s\in\cc_{\Theta}(Q).
\end{align*}
By assumption $\Theta$ is a fully invariant congruence on $(\wpf
F_{\typ}(X),\Om,\cup)$. By Lemma \ref{sec6:lem2}, for any
$P_1,\ldots,P_n\in \wpf F_{\typ}(X)$, we have
\begin{equation*}
s(P_1,\ldots,P_n)\subseteq\cc_\Theta(\{q(P_1,\ldots,P_n)\mid q\in
Q\}),
\end{equation*}
which is equivalent to the following
\begin{equation*}
s(P_1,\ldots,P_n)\cup\{q(P_1,\ldots,P_n)\mid q\in
Q\}\mathrel{\Theta} \{q(P_1,\ldots,P_n)\mid q\in Q\}.
\end{equation*}
By Theorem \ref{sec5:lem8} one obtains
\begin{equation*}
(s(P_1,\ldots,P_n)\cup\{q(P_1,\ldots,P_n)\mid q\in
Q\})^{\widetilde{\Theta}}\mathrel{\delta_\Theta}
\{q(P_1,\ldots,P_n)\mid q\in Q\}^{\widetilde{\Theta}}.
\end{equation*}
It is an immediate observation that
$t(P_1,\ldots,P_n)^{\widetilde{\Theta}}=t(P_1^{\widetilde{\Theta}},\ldots,P_n^{\widetilde{\Theta}})$
for any $t\in F_{\typ}(X)$, so one has
\begin{align*}
&s(P_1^{\widetilde{\Theta}},\ldots,P_n^{\widetilde{\Theta}})\cup\{q(P_1^{\widetilde{\Theta}},
\ldots,P_n^{\widetilde{\Theta}})\mid q\in
Q\}\mathrel{\delta_\Theta}\{q(P_1^{\widetilde{\Theta}},
\ldots,P_n^{\widetilde{\Theta}})\mid q\in
Q\}\mathrel{\Leftrightarrow}\\
&s(P_1^{\widetilde{\Theta}},\ldots,P_n^{\widetilde{\Theta}})\subseteq\cc_{\delta_\Theta}(\{q(P_1^{\widetilde{\Theta}},
\ldots,P_n^{\widetilde{\Theta}})\mid q\in Q\}).
\end{align*}
\end{proof}

\begin{thm}\label{sec4:lem5} Let $\alpha\in
\Con_{\fic}(F_{\typ}(X))$ and $\cc\in
\Clo_{\fic}^{\mathbf{0}}(F_{\typ_{\alpha}}(X))$. Then
\begin{equation*}
\cc_{\delta_{\Delta_{\Upsilon_{\cc}}}}=\cc.
\end{equation*}
Conversely, let $\Theta\in \Con_{\fic}(\wpf F_{\typ}(X))$. Then
\begin{equation*}
\Delta_{\Upsilon_{\cc_{\delta_{\Theta}}}}=\Theta.
\end{equation*}
\end{thm}
\begin{proof}
Let $\cc\in \Clo_{\fic}^{\mathbf{0}}(F_{\typ_{\alpha}}(X))$ and
$\Upsilon_{\cc}\subseteq\wpf F_{\typ_\alpha}(X)\times\wpf
F_{\typ_\alpha}(X)$ be the congruence associated with $\cc$. By
\eqref{sec4:eq2}, $\Upsilon_{\cc}=\delta_{\Delta_{\Upsilon_{\cc}}}$.
Further, by Theorem \ref{sec5:lem2} we have
\begin{equation*}
\cc=\cc_{\Upsilon_{\cc}}=\cc_{\delta_{\Delta_{\Upsilon_{\cc}}}}.
\end{equation*}

Let $\Theta\in \Con_{\fic}(\wpf F_{\typ}(X))$ and
$\cc_{\delta_{\Theta}}$ be the closure operator determined by the
congruence $\delta_{\Theta}$. By Theorem \ref{sec5:lem2} we obtain
$\delta_{\Theta}=\Upsilon_{\cc_{\delta_{\Theta}}}$ which together
with \eqref{sec4:eq1} imply
\begin{equation*}
\Theta=\Delta_{\delta_{\Theta}}=\Delta_{\Upsilon_{\cc_{\delta_{\Theta}}}}.
\end{equation*}
\end{proof}

Summarizing, Corollary \ref{sec5:cor1} and Theorem \ref{sec4:lem5} give the following.
\begin{cor} \label{sec4:cor5}
Let $\Theta\in \Con_{\fic}(\wpf F_{\typ}(X))$. There is a one-to-one
correspondence between the following sets:
\begin{itemize}
\item[(i)] the set of
all $\typ_{\widetilde{\Theta}}$-preserved
subvarieties of $\mm_{\typ_{\widetilde{\Theta}}}$;
\item[(ii)] the set $\Con_{\fic}^{\mathbf{0}}(\wpf F_{\typ_{\widetilde{\Theta}}}(X))$;
\item[(iii)] the set $\Clo_{\fic}^{\mathbf{0}}(F_{\typ_{\widetilde{\Theta}}}(X))$.
\end{itemize}
\end{cor}

\begin{exm}\label{sec6:exm6}
Let $\mathcal{SL}$ denote the variety of all semilattices and  let
$\mm_{\mathcal{SG}}$ be the variety of all semilattice ordered
semigroups. It is well known that one of its subvariety is the
variety $\mm_{\mathcal{SL}}$ of semilattice ordered semilattices $(A,\cdot,+)$.
The free semilattice $(F_{\mathcal{SL}}(X),\cdot)$ over a set $X$ is
isomorphic to the algebra $(\wpf X,\cup)$.

M. Ku\v{r}il and L. Pol\'{a}k (\cite{KP05}) found all algebraic $\varnothing$-preserving closure
operators on $\wpf X$ with properties
\eqref{sec5:eq1}, \eqref{sec6:eq1}, \eqref{sec5:eq5} and \eqref{sec4:eq4}. They showed that there are exactly four such
operators: for any $\emptyset\neq T=\{t_1,\ldots,t_k\}\subseteq\wpf
X$
\begin{align*}
&r\in\cc_1(T)\;\; \Leftrightarrow \;\; \exists(s_1,\ldots,s_j\in
T)\; \; r= s_{1}\cup\ldots\cup s_{j},\\
&r\in\cc_2(T)\;\; \Leftrightarrow \;\; \exists(t_j\in T)\;\;t_j\subseteq r \subseteq t_{1}\cup\ldots\cup t_{k},\\
&r\in\cc_3(T)\;\; \Leftrightarrow \;\; r \subseteq t_{1}\cup\ldots\cup t_{k},\\
&r\in\cc_4(T)\;\; \Leftrightarrow \;\; \exists(t_j\in
T)\;\;t_j\subseteq r.
\end{align*}

This shows that there are four $\mathcal{SL}$-preserved subvarieties
of $\mm_{\mathcal{SG}}$ and confirms previous results of R. McKenzie
and A. Romanowska (see \cite{MR79}) that there are exactly four
non-trivial subvarieties of the variety $\mm_{\mathcal{SL}}$: the
variety $\mm_{\mathcal{SL}}$ itself, the variety $\mathcal{DQL}$ of
distributive bisemilattices defined by $x+yz=(x+y)(x+z)$, the
variety $\mathcal{DL}$ of distributive lattices and the variety $\mathcal{SSL}$ of \textit{stammered semilattices}
defined by $x\cdot y=x+y$.

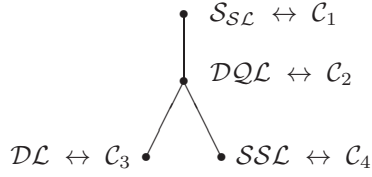
\begin{figure}[h]
\begin{center}
\setlength{\unitlength}{1mm}
\begin{picture}(34,23)(0,0)
\put(19,20){\makebox(0,0){$\mm_{\mathcal{SL}}\;\leftrightarrow\;\cc_1$}}
\put(7,20){\circle*{1}} \put(7,11){\line(0,1){9}}
\put(2,1){\circle*{1}}
\put(-8,1){\makebox(0,0){$\mathcal{DL}\;\leftrightarrow\;\cc_3$}}
\put(12,1){\circle*{1}}
\put(23,1){\makebox(0,0){$\mathcal{SSL}\;\leftrightarrow\;\cc_4$}}
\put(7,11){\circle*{1}}
\put(20,12){\makebox(0,0){$\mathcal{DQL}\;\leftrightarrow\;\cc_2$}}
\put(2,1){\line(1,2){5}} \put(12,1){\line(-1,2){5}}
\end{picture}
\end{center}
\caption{$\mathcal{SL}$-preserved subvarieties of $\mm_{\mathcal{SG}}$}\label{F:2}
\end{figure}
\end{exm}

In \cite{PZ12a} we proved that for each $\Theta\in \Con_{\fic}(\wpf
F_{\typ}(X))$ the set of all $\typ_{\widetilde{\Theta}}$-preserved
subvarieties of $S_{\typ_{\widetilde{\Theta}}}$ is a complete
join-semilattice. Now we can characterize this semilattice by
closure operators.

\begin{thm}\label{sec5:thm14}
Let $\Theta\in \Con_{\fic}(\wpf F_{\typ}(X))$. The complete
semilattice of all $\typ_{\widetilde{\Theta}}$-preserved
subvarieties of $\mm_{\typ_{\widetilde{\Theta}}}$ is
isomorphic to the semilattice
$(\Clo_{\fic}^{\mathbf{0}}(F_{\typ_{\widetilde{\Theta}}}(X)),\leq)$.

Let $I$ be a set and for each $i\in I$, $\cc_i\in \Clo_{\fic}^{\mathbf{0}}(F_{\typ_{\widetilde{\Theta}}}(X))$. The closure operator
\begin{equation*}
\cc_{\bigcap\limits_{i\in I}\Upsilon_{\cc_i}}:\p
F_{\typ_{\widetilde{\Theta}}}(X)\to \p
F_{\typ_{\widetilde{\Theta}}}(X)
\end{equation*}
is the least upper bound of $\{\cc_{i}\}_{i\in I}$ with respect to
$\leq$.
\end{thm}
\begin{proof}
Since all fully invariant congruences on $(\wpf
F_{\typ_{\widetilde{\Theta}}}(X),\Om,\cup)$ form a complete lattice
and by Corollary \ref{sec5:cor2} and Lemma \ref{sec5:lem9}, we
obtain that $\cc_{\bigcap\limits_{i\in I}\Upsilon_{\cc_i}}$ is an
algebraic $\varnothing$-preserving closure operator which satisfies the conditions
\eqref{sec5:eq1}, \eqref{sec6:eq1} and \eqref{sec5:eq5}. Finally,
using similar methods as in the proof of Lemma \ref{sec4:lem10}, one
can show that it also satisfies the condition \eqref{sec4:eq4}.
\end{proof}

\begin{exm}\cite{KP05}\label{sec5:ex1}
Let $(S,\cdot)$ be a semigroup, $S^1:=S\cup\{1\}$, where $1$ is a
neutral element, and $r\in \mathbb{N}$. A subset $A\subseteq S$ is
said to be \emph{$r$-closed} if it satisfies the following
condition: for each $p,q\in S^1$ and $u_1,u_2,\ldots,u_r\in S$
\begin{equation*}
pu_1q,pu_2q,\ldots,pu_rq\in A \;\; \Rightarrow\;\; pu_1u_2\ldots u_rq\in A.
\end{equation*}
Denote by $[X]_r$ the $r$-closed subset generated
by a non-empty set $X\subseteq A$, i.e. the intersection of all
non-empty  $r$-closed subsets of $S$ that include $X$. For $X=\emptyset$, let $ [\emptyset]_r:=\emptyset$.

The mapping
\begin{equation*}
\cc_r\colon\p S\to \p S,\;\; X\mapsto [X]_r
\end{equation*} is a closure operator defined on $S$.

Let $\typ$ be the variety of all groupoids and $\mm_r\subseteq\typ$
be the variety of all semigroups satisfying the identity $x^r\approx
x$. It was showed in \cite{KP05} that $\cc_r$ is the greatest
element in $\Clo_{\fic}^{\mathbf{0}}(F_{\mm_r}(X))$, with respect to the order
$\leq$.
\end{exm}

\section{Idempotent subvarieties}\label{sec6}

Theorems \ref{sec6:thm1}, \ref{sec6:thm2} and Corollary
\ref{sec6:cor2} are very general and they don't give a
straightforward description of the variety lattice of some
$\typ$-semilattice ordered algebras. But Theorem \ref{sec5:thm14} allows us to simplify the problem
restricting it to some particular subvarieties $\vv$ of $\typ$. For example, it
is of big interest to study idempotent subvarieties of semilattice
ordered $\Om$-algebras. Recall that a variety $\vv$ is called
\emph{idempotent} if every algebra $(A,\Om)$ in $\vv$ is idempotent,
i.e. satisfies the identities $\om(x,\ldots,x)\= x$ for every
$\om\in \Om$. So, idempotency is not a linear identity and
the variety of semilattice ordered idempotent algebras needn't be
idempotent, although this is a necessary condition as shown in
\cite{PZ12}.

In general, very little is known about idempotent subvarieties of
the variety of all semilattice ordered $\typ$-algebras. In 2005
S.~Ghosh, F.~Pastijn, X.Z.~Zhao described in \cite{GPZ05} and
\cite{P05} the lattice of all subvarieties of the variety generated
by all ordered idempotent semigroups (\emph{ordered bands}).
In the same time M.~Ku\v{r}il and L.~Pol\'{a}k
showed that the closure operator $\cc_2$ described in Example \ref{sec5:ex1}
is responsible for the idempotent subvarieties of the variety generated
by ordered bands.

A similar characterization can be obtained for the idempotent
subvarieties of the variety generated by semilattice ordered idempotent $n$-semigroups.

\begin{de} Let $2\leq n\in \mathbb{N}$. An algebra $(A,f)$ with
one $n$-ary operation $f$ is called an \emph{$n$-semigroup}, if the
following associative laws hold:
\begin{multline*}
  f(f(a_1, \dots, a_n), a_{n+1}, \dots, a_{2n-1}) \\
  = \cdots
  = f(a_1, \dots, a_i, f(a_{i+1}, \dots, a_{i+n}), a_{i+n+1}, \dots, a_{2n-1})
\\
   = \cdots = f(a_1, \dots, a_{n-1}, f(a_n, \dots, a_{2n-1})),
\end{multline*}
for all $a_1, \dots, a_{2n-1} \in A$. A semigroup is a $2$-semigroup in this sense.
\end{de}
Let $(A,f)$ be an $n$-semigroup. A subsemigroup $B\subseteq A$ of $(A,f)$ is said to be \emph{closed $n$-semigroup} if it
satisfies the following condition: \\
for each
$p_1,\ldots,p_{n-1},q_1,\ldots,q_{n-1},u_1,u_2,\ldots,u_n\in A$
\begin{align*}
&f(f(p_1,\ldots,p_{n-1},u_1),q_1,\ldots,q_{n-1}),\ldots,f(f(p_1,\ldots,p_{n-1},u_n),q_1,\ldots,q_{n-1})\in B
\;\; \Rightarrow\\
&f(f(p_1,\ldots,p_{n-1},f(u_1,\ldots,u_n)),q_1,\ldots,q_{n-1})\in B.
\end{align*}

Denote by
$[X]$ the closed $n$-semigroup generated by a non-empty set
$X\subseteq A$, i.e. the intersection of all non-empty closed
$n$-semigroups of $A$ that include $X$. For $X=\emptyset$, let $[\emptyset]:=\emptyset$. Since each $n$-semigroup
$(A,f)$ is of course a closed $n$-semigroup and the intersection of any non-empty family of
closed $n$-semigroups is a closed $n$-semigroup, 
the mapping
\begin{equation*}
\cc_{I_n}\colon\p A\to \p A,\;\; X\mapsto [X]
\end{equation*}
is a closure operator defined on $A$.

Let $(A,f)$ be an idempotent $n$-semigroup and $X\subseteq A$.
Let us define sets $X^{[k]}$ by the following recursion:
\begin{align*}
&X^{[0]}:= X,\\
&X^{[k+1]}:=\{f(f(p_1,\ldots,p_{n-1},f(u_1,\ldots,u_n)),q_1,\ldots,q_{n-1})\mid \\
&f(f(p_1,\ldots,p_{n-1},u_1),q_1,\ldots,q_{n-1}),\ldots,f(f(p_1,\ldots,p_{n-1},u_n),q_1,\ldots,q_{n-1})\in
X^{[k]}\}.
\end{align*}

It is clear that for any $k\in \mathbb{N}$, $X\subseteq
X^{[k]}\subseteq X^{[k+1]}$ and for $X\subseteq Y\subseteq A$, one
has $X^{[k]}\subseteq Y^{[k]}$. Moreover,
\begin{equation*}
[X]= \bigcup\limits_{k \in \mathbb{N}}X^{[k]}.
\end{equation*}
This implies that for $X_1,\ldots,X_n\subseteq A$, we obtain
\begin{equation*}
f([X_1],\ldots,[X_n])\subseteq [f(X_1,\ldots,X_n)].
\end{equation*}

Let $\p_\cc A:=\{[X]\mid \emptyset\neq X\subseteq A\}$ be the
set of all non-empty closed $n$-semigroups of $(A,f)$. 
By Example \ref{sec3:exm4}, the algebra $(\p_\cc A,f_\cc,+)$ is a semilattice
ordered $n$-semigroup with operations
defined in the following way
\begin{equation*}
f_\cc([X_1],\ldots,[X_n]):=[f(X_1,\ldots,X_n)],
\end{equation*}
and
\begin{equation*}
[X_1]+[X_2]:=[X_1\cup X_2].
\end{equation*}
Moreover, this algebra is idempotent. The algebra $(\p_{fin_\cc} A,f_\cc,+)$, where $\p_{fin_\cc} A$ is
the set of all non-empty, finitely generated closed $n$-semigroups of $(A,f)$,
is a subalgebra of $(\p_\cc A,f_\cc,+)$.

Let $\mathcal{A}_{I_n}$ be the variety of all idempotent $n$-semigroups.
\begin{thm}[For $n=2$, \cite{Z02}]\label{thm:zhao}
The free algebra in $\mm_{\mathcal{A}_{I_n}}$ over a set $X$ is isomorphic to the semilattice ordered $n$-semigroup of finitely generated closed $n$-semigroups
of the free $\mathcal{A}_{I_n}$-algebra over $X$.
\end{thm}

\begin{proof}
Let $(A,f,+)\in \mm_{\mathcal{A}_{I_n}}$. Then,
$(A,f)\in \mathcal{A}_{I_n}$ and any mapping $h\colon X\to A$
may be uniquely extended to an $f$-homomorphism
$\overline{h}\colon(F_{\mathcal{A}_{I_n}}(X),f)\to (A,f)$.

Further, homomorphism $\overline{h}$ can be extended to a
unique $\{f,+\}$-homomorphism
\begin{equation*}
\overline{\overline{h}}\colon(\p_{fin_\cc} F_{\mathcal{A}_{I_n}}(X),f_\cc,+)\to (A,f,+);\;\;\;\;
\overline{\overline{h}}([T])\mapsto\sum\limits_{t\in
T}\overline{h}(t),
\end{equation*}
where $T$ is a finite subset of $F_{\mathcal{A}_{I_n}}(X)$.
\end{proof}

\begin{cor}
The lattice $(\mathcal{L}(\mm_{\mathcal{A}_{I_n}}),\subseteq)$ is isomorphic to
the lattice \\ $(\{\cc\in \Clo_{\fic}^{\mathbf{0}}(F_{\mathcal{A}_{I_n}}(X))\mid \cc\leq\cc_{I_n}\},\leq)$.
\end{cor}

\begin{exm}\label{ex:nb} Let $A\in \mathcal{A}_{I_n}$ be an
entropic idempotent $n$-semigroup. Then each subalgebra of $A$
is a closed $n$-semigroup, 
and for each $X\subseteq A$, one has
$[X]=\langle X\rangle$. Denote by $\mathcal{E}_{I_n}$ the
variety of all entropic and idempotent $n$-semigroups. By
Theorem \ref{thm:zhao} the free algebra in $\mm_{\mathcal{E}_{I_n}}$
over a set $X$ is isomorphic to the semilattice ordered $n$-semigroup of finitely generated subalgebras of the free
$\mathcal{E}_{I_n}$-algebra over
$X$. This is also a special case of
Theorem \ref{thm:rs} formulated below.
\end{exm}

Algebras in the variety $\mm_{\mathcal{E}_{I_n}}$ are examples
of \emph{modals} - semilattice ordered \emph{modes}: idempotent and
entropic algebras. Another examples of modals include distributive
lattices, dissemilattices (\cite{MR79}, see also Example 5.10) - algebras $(M,\cdot,+)$
with two semilattice structures $(M,\cdot)$ and $(M,+)$ in which the
operation $\cdot$ distributes over the operation $+$, and the
algebra $(\mathbb{R},\underline{I}^0,\max)$ defined on the set of
real numbers, where $\underline{I}^0$ is the set of the following
binary operations: $\underline{p}:\mathbb{R}\times
\mathbb{R}\rightarrow \mathbb{R}$; $(x,y)\mapsto (1-p)x+py$, for
each $p\in (0,1)\subset \mathbb{R}$.

Modals were introduced and investigated
in detail by A. Romanowska and J.D.H. Smith (\cite{RS85},
\cite{RS02}). As it was indicated in \cite{R05}, only the variety of
dissemilattices (\cite{MR79}) and the variety of entropic modals
(semilattice modes \cite{K95}), are well described.

Let us consider the variety $\mathcal{M}$ of all $\Om$-modes
and let $\mm_{\mathcal{M}}$ be the variety of all modals.
\begin{thm}\cite{RS02}\label{thm:rs}
The free algebra in $\mm_{\mathcal{M}}$ over a set $X$ is isomorphic to
the modal of finitely generated non-empty subalgebras of the free mode $(F_{\mathcal{M}}(X),\Om)$ over
$X$.
\end{thm}
This implies that the mapping
\begin{equation*}
\cc_{<,>}:\p F_{\mathcal{M}}(X)\to \p F_{\mathcal{M}}(X); \;\; Y\mapsto\langle Y\rangle,
\end{equation*}
where $\langle Y\rangle$ denotes the subalgebra of
$(F_{\mathcal{M}}(X),\Om)$ generated by $Y\subseteq
F_{\mathcal{M}}(X)$, is of course a closure operator defined on
$F_{\mathcal{M}}(X)$, and what is more interesting, it is the greatest
element in the lattice $(\Clo_{\fic}^{\mathbf{0}}(F_{\mathcal{M}}(X)),\leq)$.

\begin{cor}
The lattice $(\mathcal{L}(\mm_{\mathcal{M}}),\subseteq)$ is isomorphic to
the lattice \\
$(\{\cc\in \Clo_{\fic}^{\mathbf{0}}(F_{\mathcal{M}}(X))\mid \cc\leq\cc_{<,>}\},\leq)$.
\end{cor}
Although very little is known about subvarieties of the variety of modals, just recently we described (see \cite{PZ12}) some class of closure operators on the set $F_{\mathcal{M}}(X)$, which determine
non-trivial modal subvarieties
(but of course not all of them).

In Example \ref{sec3:exm5} we introduced the notion of so-called
$\Gamma$-sinks and defined the closure operator $\cc_{\Gamma}$.
In Example \ref{sec6:exm2} we have shown that each such operator on
$F_{\mathcal{M}}(X)$ corresponds to fully invariant congruence on
$(\wpf F_{\mathcal{M}}(X),\Om,\cup)$. It follows that each operator
$\cc_{\Gamma}$ determines a non-trivial modal subvariety
\begin{equation*}
\mathcal{M}_{\Gamma}:={\rm HSP}((S_{\Gamma}(F_{\mathcal{M}}(X)),\Om,+)),
\end{equation*}
where $(S_{\Gamma}(F_{\mathcal{M}}(X)),\Om,+)$ is a modal of finitely generated $\Gamma$-sinks of $(F_{\mathcal{M}}(X),\Om)$.

Moreover we have the following:
\begin{thm}The ordered set $(\{\cc_{\Gamma}\mid\Gamma
\subseteq \Om\},\leq)$ is a bounded meet-semilattice
$(\{\cc_{\Gamma}\mid\Gamma \subseteq \Om\},\wedge)$, and for
any $\Gamma_1,\Gamma_2\subseteq \Om$,
\begin{equation*}
\cc_{\Gamma_1}\wedge
\cc_{\Gamma_2}=\cc_{\Gamma_1\cup\Gamma_2}.
\end{equation*}
\end{thm}
The operator $\cc_{\emptyset}=\cc_{<,>}$ is the
greatest element and
$\cc_{\Om}$  is the least element in this semilattice.
Additionally note that
\begin{equation*}
\cc_{\Gamma_1}\vee
\cc_{\Gamma_2}\leq\cc_{\Gamma_1\cap\Gamma_2},
\end{equation*}
but they do not need to be equal, and the operator $\cc_{\Gamma_1}\vee
\cc_{\Gamma_2}$ does not even have to belong to the set $\{\cc_{\Gamma}\mid\Gamma \subseteq \Om\}$. Hence, such operators give another examples of modal subvarieties.

Since the operator $\cc_{\Om}$ is the least element in the semilattice, none of the subvarieties $\mathcal{M}_{\Gamma}$ is a variety of semilattice modes.

In general the following problem is still open:
\begin{pr}\label{pr1}
Describe the lattice of all idempotent subvarieties of the variety
of all semilattice ordered $\typ$-algebras.
\end{pr}

Solution of
Problem \ref{pr1} could result in solving
\begin{pr}\label{pr2}
Describe the subvariety lattice of all modals.
\end{pr}

Problem \ref{pr2} is closely connected to the still open question on
the identities satisfied in the varieties generated by modes of
submodes, thoroughly discussed in \cite{R05} and partially solved in
\cite{PZ12b}.
\vskip 3mm
{\bf Acknowledgements.}
We are very much obliged to the referee for having carefully read the paper, as well as for his valuable comments.

\end{document}